\documentclass[12pt]{amsart}
\usepackage{color}
\usepackage{xcolor}
\usepackage{axodraw4j}
\usepackage{pstricks}
\usepackage{amssymb}
\usepackage{mathtools}
\usepackage{txfonts}
\usepackage{eucal}
\usepackage{extarrows}
\usepackage{wasysym}
\usepackage[all]{xy}

\usepackage{caption}
\usepackage{float} 
\usepackage{amsmath}
\usepackage{tikz}
\usepackage{cite}
\usepackage{graphicx}
\usepackage{float}

\usepackage{xspace}
\usepackage[a4paper,body={16.3cm,22.8cm},centering]{geometry}
\usepackage[colorlinks=true,pagebackref=true]{hyperref} 
\hypersetup{urlcolor=blue, citecolor=red , linkcolor= blue}
\usepackage{shuffle}
\font \eightrm=cmr8

\newcommand{\nc}{\newcommand}
\nc\smsc{0.8}

\nc\delete[1]{}
\nc{\mlabel}[1]{\label{#1}}  
\nc{\mcite}[1]{\cite{#1}}  
\nc{\mref}[1]{\ref{#1}}  
\nc{\mbibitem}[1]{\bibitem{#1}} 

\delete{
\nc{\mlabel}[1]{\label{#1}  
{\hfill \hspace{1cm}{\small\tt{{\ }\hfill(#1)}}}}
\nc{\mcite}[1]{\cite{#1}{\small{\tt{{\ }(#1)}}}}  
\nc{\mref}[1]{\ref{#1}{{\tt{{\ }(#1)}}}}  
\nc{\mbibitem}[1]{\bibitem[\bf #1]{#1}} 
}
\delete{
\nc{\mop}[1]{\mathop{\hbox {\rm #1} }}
\nc{\smop}[1]{\mathop{\hbox {\eightrm #1} }}
\nc{\mopl}[1]{\mathop{\hbox {\rm #1} }\limits}
\nc{\smopl}[1]{\mathop{\hbox {\eightrm #1} }\limits}
\def \restr#1{\mathstrut_{\textstyle |}\raise-8pt\hbox{$\scriptstyle #1$}}
\def \srestr#1{\mathstrut_{\scriptstyle |}\hbox to
  -1.5pt{}\raise-4pt\hbox{$\scriptscriptstyle #1$}}
\nc{\wt}{\widetilde}
\nc{\wh}{\widehat}

}

\newtheorem{theorem}{Theorem}[section]
\newtheorem{definition}{Definition}[section]

\newtheorem{corollary}{Corollary}[section]
\newtheorem{proposition}{Proposition}[section]
\newtheorem{lemma}{Lemma}[section]
\newtheorem{remark}{Remark}[section]
\newtheorem{example}{Example}[section]
\numberwithin{equation}{section}

\newcommand\alphlist{a,b,c,d,e,f,g,h,i,j,k,l,m,n,o,p,q,r,s,t,u,v,w,x,y,z}
\newcommand\Alphlist{A,B,C,D,E,F,G,H,I,J,K,L,M,N,O,P,Q,R,S,T,U,V,W,X,Y,Z}
\newcommand\getcmds[3]{\expandafter\newcommand\csname #2#1\endcsname{#3{#1}}}
\makeatletter
\@for\x:=\alphlist\do{\expandafter\getcmds\expandafter{\x}{frak}{\mathfrak}}
\@for\x:=\Alphlist\do{\expandafter\getcmds\expandafter{\x}{frak}{\mathfrak}}
\makeatother

\nc{\bfk}{{\bf k}}
\nc{\sha}{\shuffle}

\nc{\id}{\mathrm{id}}
\nc{\Id}{\mathrm{Id}}
\nc{\lbar}[1]{\overline{#1}}
\nc{\ot}{\otimes}
\nc{\dep}{\mathrm{dep}}
\nc{\ver}{\mathrm{ver}}

\nc{\tred}[1]{\textcolor{red}{#1}} \nc{\tgreen}[1]{\textcolor{green}{#1}}
\nc{\tblue}[1]{\textcolor{blue}{#1}} \nc{\tpurple}[1]{\textcolor{purple}{#1}}
\nc{\tcyan}[1]{\textcolor{cyan}{#1}} 
\nc{\tblk}[1]{\textcolor{black}{#1}}

\nc{\li}[1]{\tpurple{\underline{Li:}#1 }}
\nc{\liadd}[1]{\tpurple{#1}}
\nc{\xing}[1]{\tblue{\underline{Xing:}#1 }}
\nc{\yuan}[1]{\tred{\underline{Yuan:}#1 }}
\nc{\markus}[1]{\tred{\underline{Markus:} #1}}
\nc{\dominique}[1]{\tpurple{\underline{Dominique: }#1 }}
\long\def\ignore#1{}

\usetikzlibrary{calc,shapes.geometric}
\tikzset{
baseon/.style={baseline={($(#1)+(0,-0.58ex)$)}},
baseon/.default=current bounding box.center,
every picture/.style=baseon,
lst/.style={},
dst/.style={circle,inner sep=1pt,outer sep=0pt,fill,draw,dst2},
dst2/.style={fill=white},
ddst/.style={diamond,draw,inner sep=1pt},
eest/.style={ellipse,draw,inner sep=1pt,minimum size=2ex},
}

\def\zzz#1`#2...#3`#4...#5`#6@{%
--++(#1)
node[dst,label={#5:$#6$},name=#2]{}
node[midway,\hbox{Aut}o,#3]{$#4$}
}
\def\ddd#1`#2`#3@{+(#1)node[ddst,name=#2]{$#3$}}
\def\eee#1`#2`#3@{+(#1)node[eest,name=#2]{$#3$}}
\def\xxx#1`#2@{node[midway,\hbox{Aut}o,inner sep=1pt,#1]{$#2$}}
\def\pp#1`#2`#3@{node[dst,label={#2:$#3$},pos=#1]{}}
\def\oo#1`#2`#3@{\path (o) node[dst,label={#2:$#3$},name=o,#1]{};}
\def\eoo#1`#2@{\node[eest,name=o,#1] at (o) {$#2$};}

\newif\ifshowjdq
\showjdqtrue
\newcommand\setXXclip[3]{%
\def\XXheight{#1}\def\XXdepth{#2}\def\XXwidth{#3}}
\setXXclip{1}{-0.5}{1.1}

\makeatletter
\newcommand\simra{\mathrel{\mathpalette\@verra\sim}}
\def\@verra#1#2{\lower.5\p@\vbox{\lineskiplimit\maxdimen \lineskip-.5\p@
\ialign{$\m@th#1\hfil##\hfil$\crcr#2\crcr\rightarrow\crcr}}}
\makeatother

\nc{\dnx}{\Delta_n A} \nc{\dx}{\Delta A} \nc{\dgp}{{\rm deg_{P}}}
\nc{\dgt}{{\rm deg_{T}}} \nc{\dg}{{\rm deg}} \nc{\ida}{ID($A$)} \nc{\tu}{\tilde{u}} \nc{\tv}{\tilde{v}}
\nc{\nr}{\calr_n} \nc{\nz}{\calz_n} \nc{\fun}{\cala_{n,d}}
 \nc{\fbase}{\calb} \nc{\LF}{\mathrm{RF}} \nc{\FFA}{\mathrm{LF}} \nc{\irr}{\mathrm{Irr}}
 \nc{\result}{\bfk\mathrm{Irr}(S_n)}  \nc{\I}{I_{\mathrm{ID},n}^0}
 \nc{\nrs}{\calr_n^\star} \nc{\ii}{\mathrm{I}} \nc{\iii}{\mathrm{II}}
\nc{\intl}{{\rm int}}\nc{\ws}[1]{{#1}}\nc{\deleted}[1]{\delete{#1}}\nc{\plas}{placements\xspace}

\nc{\bim}[1]{#1}  \nc{\shaop}{\sha_{\Omega}^{+}}  \nc{\shao}{\sha_{\Omega}}
\nc{\bbim}[2]{#1 #2} \nc{\bbbim}[2]{#1,\, #2} \nc{\RBF}{{\rm RBF}}
\nc{\frb}{F_{\RB}} \nc{\shaf}{\ssha_{\tiny{\Omega}}} \nc{\sham}{\diamond_{\tiny{\Omega}}}
\nc{\lf}{\lfloor} \nc{\rf}{\rfloor} \nc{\shan}{\ssha_{\lambda}}
\nc{\rlex}{{\rm {lex}}} \nc{\bb}{\Box} \nc{\ra}{\rightarrow}
\nc{\e}{{\rm {e}}}
\nc{\DDF}{\mathrm{DD}(X,\,\Omega)}\nc{\DTF}{\mathrm{DT}(X,\,\Omega)} \nc{\DT}{\mathrm{DT}'(\Omega,\,V)}
\nc{\bra}{\mathrm{bra}} \nc{\bre}{\mathrm{bre}}
\nc{\dec}{\mathrm{dec}} \nc{\diamondw}{\diamond_{w}}
\nc{\type}{\mathrm{type}}

\nc\calt{\cal{T}(X,\,\Omega)} \nc\caltn{\cal{T}_n(X,\,\Omega)}
\nc\calta{\cal{T}_0(X,\,\Omega)}
\nc\caltb{\cal{T}_1(X,\,\Omega)}
\nc\caltc{\cal{T}_2(X,\,\Omega)}
\nc\caltd{\cal{T}_3(X,\,\Omega)}
\nc\caltm{\cal{T}_m(X,\,\Omega)}
\nc\caltx{\cal{T}(X)}
\nc\calf{\cal{F}(X,\,\Omega)}
\nc\fram{\frak{M}(\Omega,\, X)}
\nc\shaw{\sha^{NC}_w(\Omega,\, X)}
\nc\dw{\diamond_w} \nc\dl{\diamond_\ell}
\nc\shal{\sha^{NC}_\ell(X,\, \Omega)} \nc\shav{\sha^{NC}_w(\Omega,\, V)} \nc\shat{\sha^{NC,1}_w(\Omega,\, T^{+}(V))}
\nc{\cfo}{\cal{F}(X,\,\Omega)}
\nc{\sh}{\rm{Sh}}
\nc{\lar}{\varinjlim}
\def\cxo#1#2;{\cal{#1}#2\XO}
\nc\lrf[2]{B_{#2}^+(#1)}
\nc{\fd}{\mathrm{\text{typed angularly decorated planar rooted trees}}}
\nc{\rb}{\mathrm{RBFWs}} \nc{\dfw}{\mathrm{DFW{(X)}}} \nc{\tfw}{\mathrm{TFW{(X)}}}
\nc{\tfv}{\mathrm{TFW{(V)}}}

\def\Ve#1,#2,#3;{\vee_{#1,\,(#2,\,#3)}}
\def\bigv#1;#2;#3;{\bigvee\nolimits_{#1}^{#2;\,#3}}
\nc\rjt[2]{\mathrel{\mathop{\longrightarrow}\limits^{#1\hfill}_{\hfill#2}}}
\nc{\pl}{\cal{PLF}}
\nc{\tr}{\cal{RTF}}
\nc{\im}{\mathrm{Im}}
\nc{\ff}{\cal{F}_\Omega}
\nc{\tm}{T_\Omega}
\nc{\calp}{\cal{P}}
\makeatletter
\nc\dd{\@ifnextchar'{\ddA}{\ddB}}
\def\ddA'#1;{\rhd'_{#1\,}}
\def\ddB#1;{\rhd_{#1\,}}
\nc{\pbt}{\mathrm{PBT}}
\nc{\ad}{\mathrm{ad}}

\makeatother

\begin{document}

\title[Free pre-Lie algebras of finite posets]{Free pre-Lie algebras of finite posets}
\thispagestyle{empty}
\author{M. Ayadi}
\address{Laboratoire de Math\'ematiques Blaise Pascal,
CNRS--Universit\'e Clermont-Auvergne,
3 place Vasar\'ely, CS 60026,
F63178 Aubi\`ere, France, and University of Sfax, Faculty of Sciences of Sfax,
Laboratory of Applied Mathematics and Harmonic Analysis, route de Soukra,
3038 Sfax, Tunisia.}
\email{mohamed.ayadi@etu.uca.fr}

\tikzset{
			stdNode/.style={rounded corners, draw, align=right},
			greenRed/.style={stdNode, top color=green, bottom color=red},
			blueRed/.style={stdNode, top color=blue, bottom color=red}
		} 
		
	\begin{abstract}
	We first recall the construction of a twisted pre-Lie algebra structure on the species of finite connected topological spaces. Then, we construct a corresponding non-coassociative permutative (NAP) coproduct on the subspecies of finite connected $T_0$ topological spaces, i.e., finite connected posets, and we prove that the vector space generated by isomorphism classes of finite posets
	is a free pre-Lie algebra and also a cofree NAP coalgebra. Further, we give an explicit duality between the non-associative permutative product and the proposed NAP coproduct. Finally, we prove that the results in this paper remain true for the finite connected topological spaces.
	\end{abstract}
	
\keywords{ Bialgebras, bimonoids, finite topological spaces, Hopf algebras, species.}
\subjclass[2010]{16T05, 16T10, 16T15, 16T30, 06A11.}
\maketitle
\tableofcontents
	\section{Introduction}
	In this paper, we introduce a twisted non-associative
permutative algebra structure on the species of finite connected posets. We recall the definition of a non-associative permutative algebra (in short NAP algebra), and the dual definition of a non-coassociative permutative coalgebra (in short NAP coalgebra).
	A (left) NAP algebra is a vector space $V$ equipped with
a bilinear product $\cdot$ satisfying the relation:
\begin{align}\label{NAP algebra}
    a\cdot(b\cdot c)=b\cdot(a\cdot c), \hbox{ for all } a, b ·\hbox{ and }c \in V.
\end{align}
We note also that the notion of NAP algebra has emerged from the reference \cite{L. Muriel} by M. Livernet. It has also been previously defined in \cite{A. Dzhumadil} by A. Dzhumadil’daev, and C. L{\"o}fwall under the name of "left commutative algebra".
 Dually a (left) NAP coalgebra is a vector space $V$ equipped with
a bilinear coproduct $\delta$: $V\to V\otimes V$ satisfying the relation:
\begin{align}\label{NAP coalgebra}
    (Id \otimes \delta)\delta=\tau^{12}(Id \otimes \delta)\delta.
\end{align}
We recall the following rigidity theorem \cite{L. Muriel}: any pre-Lie algebra $(V, \rhd)$, together with a non-associative permutative connected coproduct $\delta$ satisfying the distributive law
\begin{equation}\label{equ}
   \delta(a\rhd b)=a\otimes b+(a\otimes \mathbf 1+\mathbf 1\otimes a)\rhd \delta(b), 
\end{equation}
is a free pre-Lie algebra and a cofree NAP coalgebra.\\

 We recall also in this paper the description of the free pre-Lie algebra in terms of rooted trees given in \cite{acg5}. We define a new bilinear product $\circleright$ in the species of finite connected posets $\mathbb{U}$ by: for all $P\in \mathbb{U}_{X_1}$ and $Q\in \mathbb{U}_{X_2}$, where $X_1$ and $X_2$ are two finite sets
	 $$P\circleright Q=\sum_{v\in \tiny{\hbox{min}}(Q)}P\searrow_v Q,$$
	  where  $P\searrow_v Q$ is
obtained from the Hasse graphs $G_1$ and $G_2$ of $P$ and $Q$ by adding an (oriented)
edge from $v$ in $G_2$ to any minimal vertex of $G_1$\cite{Moh. twisted}. \\

	 We show that the  bilinear product $ \circleright $ verifies \eqref{NAP algebra} in the monoidal category of species, and therefore endows the species $\mathbb U$ with a NAP algebra structure.\\
	 
	 
	 For any finite set $X$, we define the coproduct $\delta$ by:
	\begin{eqnarray*}
		\delta:\mathbb{U}_X&\longrightarrow& (\mathbb{U}\otimes \mathbb{U})_X=\bigoplus_{Y\sqcup Z= X}\mathbb{U}_{Y}\otimes\mathbb{U}_{Z}\\
		P&\longmapsto& \frac{1}{|min(P)|}\sum_{I\circledcirc P} I \otimes P\backslash I,
	\end{eqnarray*}
	where $I\circledcirc P$ means that $I$ is a subset of $P$ such that:
	\begin{itemize}
	\item $I_-$ is a singleton included in $\hbox{min}(P)$,
    \item and $I$ is a connected component of the set $\{x\in P, I_-<_P x \}$.
\end{itemize}
with the set $I_-$ is equal to the space $ \{x\notin I, \hbox{ there exists } y \in I \hbox{ such that } x \le_{P} y \} $. 

	We prove that the coproduct $\delta$ is a twisted NAP coproduct, i.e., \eqref{NAP coalgebra} is verified in the monoidal category of species.\\

M. Livernet in \cite{L. Muriel}
proved the following rigidity theorem: for any pre-Lie algebra $(V, \rhd)$ together with a NAP connected coproduct $\delta$ satisfying relation \eqref{equ}, $V$ is the free pre-Lie algebra generated by $\hbox{Prim}(V)$.\\

	 Later in this paper, we prove a compatibility relation between the pre-Lie $\searrow$ structure and the coproduct $\delta$ by proving the twisted version of \eqref{equ} namely: 
	 $$\delta(P\searrow Q)=P\otimes Q+(P\otimes \mathbf 1+\mathbf 1\otimes P)\searrow \delta(Q),$$
	 for any pair $(P, Q)$ of finite connected posets, where the unit $\mathbf 1$ is identified to the empty poset.

	 Applying Aguiar-Mahajan's bosonic Fock functor $\overline{\mathcal{K}}$ \cite{acg10} and M. Livernet's rigidity theorem \cite{L. Muriel} leads to the main result of the paper: $\big(\overline{\mathcal{K}}(\mathbb{U}), \searrow \big)$
		 endowed with the coproduct $\delta$ is a free pre-Lie algebra and a cofree NAP coalgebra.
		 
		We end up Section \ref{Free} by proving that the NAP product $\circleright$ and the  NAP coproduct $\delta$ are dual to each other, and we show that $(\overline{\mathcal{K}}(\mathbb{P}), \Delta_{\searrow}) $ is a coassociative cofree coalgebra, where $\Delta_{\searrow}$ is the coassociative coproduct defined in \cite{Moh. twisted}. \\ 
		 
		In the last section, we prove that the results in this paper remain true for the finite connected topological spaces, with a small change on the definition of the coproduct $\delta$.
\section{Basics of finite topologies}
A partial order on a set $X$ is a transitive, reflexive and antisymmetric relation on $X$.
A finite poset is a finite set $X$ endowed with a partial order $\le$.
Let $P=(X, \le_{P})$ and $Q=(X, \le_{Q})$ be two posets. We say that $P$ is finer than $Q$ if:
$x\le_{P}y \Longrightarrow x\le_{Q}y$ for any $x, y\in X$.
 The Hasse diagram of poset $P=(X, \le_{P})$
is obtained by representing any element of $X$ by a vertex, and by drawing a directed edge from $a$
to $b$ if and only if $a<_{P} b$, and, for any $c\in X$ such that $a\le_{P} c\le_{P} b$,
one has $a=c$ or $b=c$.
We can say that $I$ is an upper ideal of $P$ $(I\subset P)$ if, for all $x, y\in P$, $(x\in I, x\le_P y)\Longrightarrow y\in I$. We denote $J(P)$ the set of upper ideals of $P$.\\

Recall \cite{P. Alex, F. L. D.} that a finite topological space is a finite quasi-poset and vice versa.
	Any topology $\mathcal{T}$ (hence any quasi-order on $X$) gives rise to an equivalence relation:
	\begin{equation}
	x \sim_{\mathcal{T}}y\Longleftrightarrow \left( x\leq_{\mathcal{T}}y \hbox{ and } y\leq_{\mathcal{T}}x \right).
	\end{equation}
Equivalence classes will be called bags here. This equivalence relation is trivial if and only if the quasi-order is a (partial) order, or equivalently, if the corresponding topology is $T_0$. Any
topology $\mathcal{T}$ on $X$ defines a poset structure on the quotient $X/\sim_{\mathcal{T}}$, corresponding to the partial order induced by the quasi-order $\le_{\mathcal{T}}$. More on finite topological spaces can be found in \cite{J. A., R. E., Moh. Doubling, F. L. D.}.\\

Recall \cite{acg...3, acg10} that a linear (tensor) species is a contravariant functor from the category
	of finite sets $\mathbf{Fin}$ with bijections into the category $\mathbf{Vect}$ of vector spaces (on some field $\mathbf{k}$). 
	Let $\mathbb{E}$ be a linear species. We note $\overline{\mathcal{K}}(\mathbb{E})=\bigoplus \limits_{\underset{}{n\geq 0}} \mathbb{E}_n/S_n$, where  $\mathbb{E}_n/S_n$ denotes the space of $S_n$-coinvariants of $\mathbb{E}_n$.
	The functor $\overline{\mathcal{K}}$ from linear species to graded vector spaces is intensively studied in \cite[chapter 15]{acg10} under the name "bosonic Fock functor". \\
	
	The species $\mathbb{T}$ of finite topological spaces is defined as follows: for any finite set $X$, $\mathbb{T}_X$  is
	the vector space freely generated by the topologies on $X$. For any bijection
	$\varphi : X \longrightarrow  X^{\prime } $, the isomorphism $\mathbb{T}_{\varphi } : \mathbb{T}_{ X^{\prime}} \longrightarrow \mathbb{T}_X$ is defined
	for any topology $\mathcal{T}$ on $X^{\prime }$,
	by the following
	relabelling:
	$$\mathbb{T}_{\varphi }(\mathcal{T})=\{ \varphi^{-1}(Y), Y\in  \mathcal{T}  \}$$ 
	\\
	It is well-known that the species $\mathbb{T}$ of finite topological spaces is a
twisted commutative Hopf algebra \cite{F. L. D.}: the product $m$ is given by disjoint union, and the coproduct $\Delta$ is a natural generalization of the Connes-Kremier coproduct on rooted forests.
We defined another twisted Hopf algebra structure in \cite{Moh. twisted} by replacing $\Delta$ with the coproduct
\begin{eqnarray*}
		\Delta_{\searrow}:\mathbb{T}_X&\longrightarrow& (\mathbb{T}\otimes \mathbb{T})_X=\bigoplus_{Y\sqcup Z= X}\mathbb{T}_{Y}\otimes\mathbb{T}_{Z}\\
		\mathcal{T}&\longmapsto& \sum_{Y \overline{\in}  \mathcal{T}} \mathcal{T}_{|Y} \otimes \mathcal{T}_{|X\backslash Y}.
	\end{eqnarray*}
			Where $Y \overline{\in}  \mathcal{T}$, stands for
\begin{itemize}
    \item $Y\in \mathcal{T}$,
    \item $\mathcal{T}_{|Y}=\mathcal{T}_1...\mathcal{T}_n$, such that for all $i\in \{1,..., n\}, \mathcal{T}_i$ is connected and \big($\hbox{min}\mathcal{T}_i=(\hbox{min}\mathcal{T})\cap \mathcal{T}_i$, or there is a single common ancestor $x_i \in \overline{X\backslash Y}$ to $\hbox{min}\mathcal{T}_i$\big), where $\overline{X\backslash Y}=(X\backslash Y)/\sim_{\mathcal{T}_{|X\backslash Y}}$.
\end{itemize}
If we apply the functor $\overline{\mathcal{K}}$, we therefore two commutative graded connected Hopf algebras
$\overline{\mathcal{K}}(\mathbb{T}, m, \Delta)$ and $\overline{\mathcal{K}}(\mathbb{T}, m, \Delta_{\searrow})$.
	\begin{definition}
    \cite{acg4, acg3} A left pre-Lie algebra over a field $\mathbf{k}$ is a $\mathbf{k}$-vector space $A$ with a binary composition $\rhd$ that satisfies the left pre-Lie identity:
    \begin{equation}
        (x\rhd y)\rhd z-x\rhd (y\rhd z)=(y\rhd x)\rhd z-y\rhd (x\rhd z),
    \end{equation}
    for all $x, y, z\in A$. The left pre-Lie identity rewrites
as:
\begin{equation}
L_{[x,y]}=[L_x, L_y],
\end{equation}
where $L_x:A\longrightarrow A$ is defined by $L_xy=x\rhd y$, and where the bracket on the left-hand side is
defined by $[x,y]=x\rhd y-y\rhd x$. As a consequence this bracket satisfies the Jacobi identity.
\end{definition}
In fact, as mentioned in \cite{Moh. twisted}, the structure $\searrow$ on the species of connected finite topological spaces $\mathbb{V}$ defined by: for all $\mathcal{T}_1=(X_1, \leq_{\mathcal{T}_1})$ and $\mathcal{T}_2=(X_2, \leq_{\mathcal{T}_2})$ two  finite connected topological spaces as: 
$$ \mathcal{T}_1\searrow \mathcal{T}_2=\sum_{v\in X_2}\mathcal{T}_1\searrow_v \mathcal{T}_2,$$
 is a pre-Lie structure, where $\mathcal{T}_1\searrow_v \mathcal{T}_2$ is
obtained from the Hasse graphs $G_1$ and $G_2$ of $\mathcal{T}_1$ and $\mathcal{T}_2$ by adding an (oriented)
edge from $v$ in $G_2$ to any minimal vertex of $G_1$. \\

	We denote by $\mathbb{P}$ the sub-species of $\mathbb{T}$ consisting in partial order (i.e., $T_0$-topologies).
The binary product $m$ (resp. $\searrow$) restricts to finite posets (resp. finite connected posets), as well as both coproducts $\Delta$ (resp. $\Delta_{\searrow}$). Hence the triple
$(\mathbb{P}, m, \Delta)$ is a twisted Hopf subalgebra of $(\mathbb{T}, m, \Delta)$, the triple $(\mathbb{P}, m, \Delta_{\searrow})$ is a twisted Hopf subagebra of $(\mathbb{T}, m, \Delta_{\searrow})$ and $(\mathbb{P}, \searrow)$ is a twisted pre-Lie subalgebra of  $(\mathbb{T}, \searrow)$.\\

	\section{Free pre-Lie algebras and cofree coalgebras}\label{Free}
		\subsection{Free pre-Lie algebras and rooted trees}
	Let $T$ be the vector space spanned by the set of isomorphism classes of rooted trees and $H=S(T)$.
	Grafting pre-Lie algebras of rooted trees were studied for the first time by F. Chapoton and M. Livernet \cite{acg5}, although the use of rooted trees can be traced back to A. Cayley \cite{A. Cayley}. The grafting product is given, for all $t, s \in T$, by:
	\begin{equation}
	    t\rightarrow s=\sum_{s' \hbox{ \tiny{vertex of} }  s} t\rightarrow_{s'}s,
	\end{equation}
	where $t\rightarrow_{s'} s$ is the tree obtained by grafting the root of $t$ on the vertex $s'$ of $s$. In other words, the operation $t\rightarrow s$ consists of grafting the root of $t$ on every vertex of $s$ and summing up.\\
	
	We note $T(n)$ the linear span of rooted trees of degree $n$. For a vector space $V$, we denote by $T(V)$ the space $\bigoplus \limits_{n}T(n)\otimes V^{\otimes n}/S_n$. This is the linear span of rooted trees decorated by a basis of $V$. Here $S_n$ is the symmetric group on $n$ elements. 	Following the notation of Connes and Kreimer \cite{C. Alain}, any tree $t$ writes
$t:= B(v, t_1, t_2, ..., t_n)$ that is decorated by $v$, and $t_1, ..., t_n$ are trees.
	
	We notice that, if using the pre-Lie product $\rightarrow$ in $T(V)$, one has:
	$$B(v, t_1,...,t_n)=t_n \rightarrow B(v, t_1,...,t_{n-1}) - \sum_{0<k<n}B(v, t_1,..., t_n\rightarrow t_k,...,t_{n-1})$$
	\begin{proposition}\cite{acg5}
	Equipped by $\rightarrow$, the space $T$ is the free pre-Lie algebra with one generator.
	\end{proposition}
	\begin{proposition}\cite{L. Muriel}
	    Let $V$ be a vector space. Then, $T(V)$ together with the coproduct
	    $$\delta\big(B(v, t_1, t_2,...,t_n)\big)=\sum_{0<k<n+1}t_k \otimes B(v, t_1,..., \hat{t}_k,...,t_n),$$
	    is the cofree NAP connected coalgebra generated by $V$.
	\end{proposition}

	\subsection{Twisted NAP algebras of finite connected posets}
	Let $P$ be a finite connected poset, and $I$ be a subset of $P$. We denote by $I_-$ the set $\{x\notin I, \hbox{ there exists } y\in I \hbox{ such as } x\le_{P} y\}$.\\
	We note $I\circledcirc P$ whenever: 
	\begin{itemize}
	\item $I_-$ is a singleton included in $\hbox{min}(P)$,
    \item and $I$ is a connected component of the set $\{x\in P, I_-<_P x \}$.
\end{itemize}
	\begin{lemma}\label{lem}
		For any finite connected poset $P$, and any subsets $I, J$ of $P$, we have:
		 \begin{itemize}
    \item[1-] if $I\circledcirc P$ and $J\circledcirc P$, then $I=J \hbox{ or } I\cap J=\emptyset$.
    \item[2-] If $I\circledcirc P$, then $\hbox{min}(P\backslash I)=\hbox{min}(P)$.
   \item[3-] If $I\circledcirc P$, then $P\backslash I$ is a connected poset.
\end{itemize}
	\end{lemma}
	\begin{proof}
	1- Let $P$ be a finite connected poset and $I, J$ two subsets of $P$, such that $I, J\circledcirc P$. So, we have two possible cases: either $I_-\ne J_-$ or $I_- = J_-$.\\
	- If $I_-\ne J_-$, with $I_-=\{ y \}$, and $J_-=\{ z \}$, then $I$ is a connected component of the set $E_1=\{x\in P, y<_Px \}$ and $J$ is a connected component of the set $E_2=\{x\in P, z<_Px \}$. So $I$ and $J$ is a connected component of the set $E_1 \sqcup E_2$, then $I=J$ or $I\cap J=\emptyset$. Gold $I_-\ne J_-$, then $I\ne J$. Hence $I\cap J=\emptyset$.\\
	- If $I_- = J_-= \{y\}$, then $I$ and $J$ is two connected components of the set $\{x\in P, y<_Px \}$, then: $I=J$ or $I\cap J=\emptyset$.\\
	
	2- If there exists $x\in \hbox{ min }(P)\cap I$, then $I_-<_P x$, which is absurd. Thus, $\hbox{ min}(P)\cap I=\emptyset$. Then, $\hbox{min}(P\backslash I)=\hbox{min}(P)$.\\
	
	3- If $I\circledcirc P$, then $I$ is a connected component of the set $E=\{x\in P, I_-<_P x \}$.  
	We have $P\backslash I$ is a connected poset if and only if, for all $x, y \in P\backslash I$ there exist
$a_1,..., a_n \in P\backslash I$ such that $x\mathcal{R}a_1 ...\mathcal{R}a_n\mathcal{R}y$, where $\mathcal{R}$ is defined by: $t_1\mathcal{R}t_2 \iff (t_1<_P t_2 \hbox{ or } t_2<_P t_1$). We assume that $x, a_1,..., a_n, y$ is the smallest chain that connects $x$ to $y$.\\
 If $P\backslash I$ is not connected, there is $k\in [n]$ such that $a_k \in I$. By choosing $k$ appropriately we can also assume that $a_{k-1}$ or $a_{k+1}$ are not in $I$, then we have four possible cases:
\begin{itemize}
    \item First case; $a_{k-1}\le_P a_{k}\le_P a_{k+1}$.
Then this chain is not the smallest chain that connects $x$ to $y$, which is absurd.
    \item Second case; $a_{k-1}\ge_P a_{k}\ge_P a_{k+1}$, the proof is similar.
    \item Third case; $a_{k-1}\ge_P a_{k}\le_P a_{k+1}$.
    Since $a_k \in I$, then $I_- <_P a_k$. Moreover $<_P$ is transitive, then
$I_- <_P a_{k-1}$, so $a_{k-1} \in E$. Then we obtained: $(a_{k-1}, a_k)\in E\times I$, $I$ is a connected component of $E$
and $a_k <_P a_{k-1}$. Hence $a_{k-1} \in I$, and similarly $a_{k+1}\in I$.
    Then $a_{k-1}\in I$ and $a_{k+1} \in I$, which is in contradiction with the assumptions.
    \item Fourth case; $a_{k-1}\le_P a_{k}\ge_P a_{k+1}$.  In this case we have $a_{k-1} \in I_-$ and $a_{k+1}\in I_-$, so $a_{k-1}=a_{k+1}$. Then this chain is not the smallest chain that connects $x$ to $y$, which is absurd.
\end{itemize}
Hence, $P\backslash I$ is a connected poset.

	\end{proof}
	\begin{definition}
	The bilinear product $\circleright$ is defined in the species of finite connected posets $\mathbb{U}$ as follows: for all $P\in \mathbb{U}_{X_1}$ and $Q\in \mathbb{U}_{X_2}$, where $X_1$ and $X_2$ are two finite sets:
	\begin{align*}
	    P\circleright Q=\sum_{v\in \tiny{\hbox{min}}(Q)}P\searrow_v Q.
	\end{align*}
	\end{definition}
	\begin{proposition}
	  The  bilinear product $ \circleright $ endows the species $\mathbb{U}$ with a twisted NAP algebra structure i.e., the following identity is verified: 
	  $$P\circleright(Q\circleright R)=Q\circleright(P\circleright R)$$ 
	  for any three finite connected posets $P, Q$ and $R$.
	\end{proposition}
	\begin{proof}
	Let $P, Q$ and $R$ be three finite connected posets, we have:
	\begin{align*}
	    P\circleright(Q\circleright R)&=\sum_{v\in \tiny{\hbox{min}}(R)}P\circleright(Q\searrow_v R)\\
	    &=\sum_{u\in \tiny{\hbox{min}}(Q\searrow_v R)}\sum_{v\in \tiny{\hbox{min}}(R)}P\searrow_u(Q\searrow_v R),
	\end{align*}
	and since $\hbox{min}(Q\searrow_v R)=\hbox{min}(R)$, for all $v\in \hbox{min}(R)$ then:
	\begin{align*}
	    P\circleright(Q\circleright R)&=\sum_{v,\, u\in \tiny{\hbox{min}}(R)}P\searrow_u(Q\searrow_v R),
	\end{align*}
	Which is symmetric on $P$ and $Q$. Correspondingly, we obtain:
	$$P\circleright(Q\circleright R)=Q\circleright(P\circleright R).$$
	\end{proof}
	\begin{corollary}\cite[Chapter 17]{acg10}
	Applying the functor $\overline{\mathcal{K}}$ gives that $\big(\overline{\mathcal{K}}(\mathbb{U}), \circleright \big)$ is a NAP algebra.
	\end{corollary}
	\begin{definition}\label{def-connected}\cite{L. Muriel}
Let $(V, \delta)$ be a coalgebra, i.e. be a vector space $V$ together with linear map $\delta:V\longrightarrow V\otimes V$. The following defines a filtration on $V$:
\begin{itemize}
    \item $\hbox{ Prim}(V)=V_1=\{ x\in V, \delta(x)=0 \}$,
    \item $V_n=\{ x\in V, \delta(x)\in \bigoplus \limits_{0<k<n}V_k\otimes V_{n-k} \}$.
\end{itemize}
The coalgebra $(V, \delta)$ is said to be connected if $V=\bigcup \limits_{k>0} V_k$.
\end{definition}
	\begin{proposition}\label{prop}
	The coproduct $\delta:\mathbb{U}\longrightarrow \mathbb{U}\otimes \mathbb{U}$
	defined for any finite set $X$ by
	\begin{eqnarray*}
		\delta:\mathbb{U}_X&\longrightarrow& (\mathbb{U}\otimes \mathbb{U})_X=\bigoplus_{Y\sqcup Z= X}\mathbb{U}_{Y}\otimes\mathbb{U}_{Z}\\
		P&\longmapsto& \frac{1}{|min(P)|}\sum_{I\circledcirc P} I \otimes P\backslash I,
	\end{eqnarray*}
	 is a twisted connected NAP coproduct, i.e. the following identity is verified 
	$$(Id \otimes \delta)\delta=\tau^{12}(Id \otimes \delta)\delta,$$
	$\hbox{ where } \tau^{12}=\tau \otimes Id, \tau \hbox{ is the flip }.$
	\end{proposition}
	\begin{proof}\label{proof}
	Let $P$ be a finite connected poset. We have
	\begin{align*}
	    (Id \otimes \delta)\delta(P)&=\frac{1}{|min(P)|}\sum_{I\circledcirc P} I \otimes \delta(P\backslash I)\\
	    &=\frac{1}{|min(P)|}\sum \limits_{\underset{J\circledcirc P\backslash I}{ I\circledcirc P }}\frac{1}{|min(P\backslash I)|} I \otimes J\otimes \big((P\backslash I)\backslash J\big)\\
	    &=\frac{1}{|min(P)|^2}\sum \limits_{\underset{I\bigcap J=\emptyset}{ I,\, J\circledcirc P }} I \otimes J\otimes \big(P\backslash (I\sqcup J)\big)\\
	    &=\tau^{12}(Id \otimes \delta)\delta(P).
	\end{align*}
	By the Definition 2.6 in  \cite{L. Muriel} and the definition of the coproduct $\delta$, it is clear that the coalgebra $(\mathbb{U}, \delta)$ is connected. Hence, $(\mathbb{U}, \delta)$ is a twisted connected NAP coalgebra. 
	\end{proof}
	\begin{example}
	$\delta(\fcolorbox{white}{white}{
 \scalebox{0.7}{
  \begin{picture}(8,17) (217,-263)
    \SetWidth{1.0}
    \SetColor{Black}
    \Vertex(210,-260){2}
    \Vertex(226,-260){2}
    \Vertex(218,-249){2}
    \Line(210,-261)(219,-248)
    \Line(226,-260)(218,-250)
  \end{picture}
}})=\hspace*{-0.09cm} 0$
\hspace*{4cm}$\delta(\fcolorbox{white}{white}{
\scalebox{0.7}{
  \begin{picture}(8,19) (218,-274)
    \SetWidth{1.0}
    \SetColor{Black}
    \Vertex(219,-270){2}
    \Vertex(210,-258){2}
    \Vertex(226,-258){2}
    \Line(218,-271)(226,-257)
    \Line(220,-270)(209,-258)
  \end{picture}
}})=2\fcolorbox{white}{white}{
\scalebox{0.7}{
  \begin{picture}(-6,17) (242,-291)
    \SetWidth{1.0}
    \SetColor{Black}
    \Vertex(235,-288){2}
    \SetWidth{1.0}
  \end{picture}
}} \otimes
\fcolorbox{white}{white}{
\scalebox{0.7}{
  \begin{picture}(-6,17) (242,-291)
    \SetWidth{1.0}
    \SetColor{Black}
    \Line(235,-278)(235,-288)
    \SetWidth{0.0}
    \Vertex(235,-288){2}
    \SetWidth{1.0}
    \Vertex(235,-277){2}
  \end{picture}
}} $\\
\hspace*{2.57cm}$\delta(\fcolorbox{white}{white}{
\scalebox{0.7}{
  \begin{picture}(26,39) (415,-324)
    \SetWidth{1.0}
    \SetColor{Black}
    \Vertex(433,-321){2}
    \Vertex(433,-309){2}
    \Vertex(418,-308){2}
    \Vertex(419,-321){2}
    \Line(419,-321)(418,-309)
    \Line(433,-322)(417,-308)
    \Line(433,-322)(433,-309)
  \end{picture}
}})=\frac{1}{2}\fcolorbox{white}{white}{
\scalebox{0.7}{
  \begin{picture}(-6,17) (242,-291)
    \SetWidth{1.0}
    \SetColor{Black}
    \Vertex(235,-288){2}
    \SetWidth{1.0}
  \end{picture}
}}\hspace*{-0.09cm}\otimes \hspace*{-0.05cm} \fcolorbox{white}{white}{
\scalebox{0.7}{
  \begin{picture}(22,17) (218,-263)
    \SetWidth{1.0}
    \SetColor{Black}
    \Vertex(210,-260){2}
    \Vertex(226,-260){2}
    \Vertex(218,-249){2}
    \Line(210,-261)(219,-248)
    \Line(226,-260)(218,-250)
  \end{picture}
}}$
\hspace*{2cm}$\delta(\fcolorbox{white}{white}{
\scalebox{0.7}{
  \begin{picture}(11,30) (216,-263)
    \SetWidth{1.0}
    \SetColor{Black}
    \Vertex(218,-259){2}
    \Vertex(210,-247){2}
    \Vertex(226,-247){2}
    \Vertex(218,-236){2}
    \Line(218,-260)(226,-246)
    \Line(219,-260)(209,-247)
    \Line(210,-248)(219,-235)
    \Line(226,-247)(218,-237)
  \end{picture}
}})=\hspace{-0.12cm} \fcolorbox{white}{white}{
\scalebox{0.7}{
  \begin{picture}(22,17) (210,-263)
    \SetWidth{1.0}
    \SetColor{Black}
    \Vertex(210,-260){2}
    \Vertex(226,-260){2}
    \Vertex(218,-249){2}
    \Line(210,-261)(219,-248)
    \Line(226,-260)(218,-250)
  \end{picture}
}}\hspace*{-0.15cm}\otimes \fcolorbox{white}{white}{
\scalebox{0.7}{
  \begin{picture}(-6,17) (242,-291)
    \SetWidth{1.0}
    \SetColor{Black}
    \Vertex(235,-288){2}
    \SetWidth{1.0}
  \end{picture}
}}$
	\end{example}
	\begin{corollary}\cite[Chapter 17]{acg10}
	Applying the functor $\overline{\mathcal{K}}$ gives that $\big(\overline{\mathcal{K}}(\mathbb{U}), \delta \big)$ is a NAP coalgebra.
	\end{corollary}
	\subsection{Free twisted pre-Lie algebras of finite posets}
	In this subsection, we prove the compatibility relation \cite[Theorem]{L. Muriel} between the pre-Lie product $\searrow$ and the NAP coproduct $\delta$.
	\begin{theorem}\label{Theorem}
	Let $P, Q$ be two finite connected posets. We have the following identity
$$\delta(P\searrow Q)=P\otimes Q+(P\otimes \mathbf 1+\mathbf 1\otimes P)\searrow \delta(Q),$$
 where the unit $\mathbf 1$ is identified to the empty poset.
	\end{theorem}
	\begin{proof}\label{prv}
Let $P=(X_1, \leq_P)$, $Q=(X_2, \leq_Q)$ be two finite connected posets, we have:
	\begin{align*}
	    \delta(P\searrow Q)&=\sum_{v\in X_2}\delta(P\searrow_v Q)\\
	    &=\sum \limits_{\underset{I\circledcirc (P\searrow_v Q)}{v\in X_2 }}\frac{1}{|min(P\searrow_v Q)|} I \otimes \big((P\searrow_v Q)\backslash I\big)\\
	    &=\sum \limits_{\underset{I\circledcirc (P\searrow_v Q)}{v\in X_2 }} \frac{1}{|min(Q)|}I \otimes \big((P\searrow_v Q)\backslash I\big)\\
	    &=\frac{1}{|min(Q)|}\left( \sum \limits_{\underset{I\circledcirc (P\searrow_v Q)}{v\in X_2,\, v\in \hbox{\tiny{min}}(Q) }} I \otimes \big((P\searrow_v Q)\backslash I\big)+\sum \limits_{\underset{I\circledcirc (P\searrow_v Q)}{v\in X_2,\, v\notin \hbox{\tiny{min}}(Q) }} I \otimes \big((P\searrow_v Q)\backslash I\big)\right),
	\end{align*}
	we notice that\\
	- if $v\in \hbox{min}(Q)$, then
	$\big \{I, I\circledcirc (P\searrow_v Q)  \}=\{ P \}\cup \{ I,   I\circledcirc Q \big \},$\\
	and\\
	- if $v\notin \hbox{min}(Q)$, then $\{I, I\circledcirc (P\searrow_v Q)  \}=\{ P\searrow_v J, J\circledcirc Q, v\in J  \}\cup \{ J,   J\circledcirc Q, v\notin J \}.$\\
	Then 
	\begin{align*}
	    \delta(P\searrow Q)&=\frac{1}{|min(Q)|}\left( \sum \limits_{\underset{}{v\in X_2,\, v\in \hbox{\tiny{min}}(Q) }} P \otimes Q+ \sum \limits_{\underset{I\circledcirc Q}{v\in X_2,\, v\in \hbox{\tiny{min}}(Q) }} I \otimes \big(P\searrow_v (Q\backslash I)\big)\right)\\
	    &\hspace{1cm}+\frac{1}{|min(Q)|}\left(\sum \limits_{\underset{J\circledcirc Q,\, v\in J}{v\in X_2,\, v\notin \hbox{\tiny{min}}(Q) }} (P\searrow_v J) \otimes  Q\backslash J +\sum \limits_{\underset{J\circledcirc Q,\, v\notin J }{v\in X_2,\, v\notin \hbox{\tiny{min}}(Q)}} J \otimes \big(P\searrow_v (Q\backslash J)\big) \right),
	\end{align*}
	we notice that
	$$\sum \limits_{\underset{J\circledcirc Q,\, v\in J}{v\in X_2,\, v\notin \hbox{\tiny{min}}(Q) }} (P\searrow_v J) \otimes  Q\backslash J=\sum \limits_{\underset{J\circledcirc Q}{}} (P\otimes \mathbf 1)\searrow (J \otimes Q\backslash J),$$
	and
	$$\sum \limits_{\underset{J\circledcirc Q,\, v\notin J }{v\in X_2,\, v\notin \hbox{\tiny{min}}(Q)}} J \otimes \big(P\searrow_v (Q\backslash J)\big)=\sum \limits_{\underset{J\circledcirc Q }{v\in X_2,\, v\notin \hbox{\tiny{min}}(Q)}} J \otimes \big(P\searrow_v (Q\backslash J)\big),$$
	then 
	$$ \sum \limits_{\underset{I\circledcirc Q}{v\in X_2,\, v\in \hbox{\tiny{min}}(Q) }} I \otimes \big(P\searrow_v (Q\backslash I)\big)+ \sum \limits_{\underset{J\circledcirc Q,\, v\notin J }{v\in X_2,\, v\notin \hbox{\tiny{min}}(Q)}} J \otimes \big(P\searrow_v (Q\backslash J)\big)=\sum \limits_{\underset{J\circledcirc Q}{}} (\mathbf 1\otimes P)\searrow (J \otimes Q\backslash J),$$
	accordingly
	\begin{align*}
	    \delta(P\searrow Q)&=\frac{1}{|min(Q)|}\left( \sum \limits_{\underset{}{v\in X_2,\, v\in \hbox{\tiny{min}}(Q) }} P \otimes Q + \sum \limits_{\underset{J\circledcirc Q}{}} (P\otimes \mathbf 1)\searrow (J \otimes Q\backslash J) + \sum \limits_{\underset{J\circledcirc Q}{}} (\mathbf 1\otimes P)\searrow (J \otimes Q\backslash J)\right).
	\end{align*}
	Hence,
	$$\delta(P\searrow Q)=P \otimes Q + (P\otimes \mathbf 1)\searrow \delta(Q) +(\mathbf 1\otimes P)\searrow \delta(Q).$$
	\end{proof}
	\begin{corollary}
		Applying the functor $\overline{\mathcal{K}}$ gives that, $\big(\overline{\mathcal{K}}(\mathbb{U}), \searrow \big)$
		 endowed with the coproduct $\delta$ is a free pre-Lie algebra and a cofree NAP coalgebra.
	\end{corollary}
	\begin{proof}
	This is a direct consequence of Theorem \ref{Theorem} and M. Livernet's rigidity theorem \cite[Theorem]{L. Muriel}.
	\end{proof}
		\begin{corollary}
	 $\overline{\mathcal{K}}(\mathbb{U})$ is generated as a pre-Lie algebra by $\hbox{ Prim}\big(\overline{\mathcal{K}}(\mathbb{U})\big)$.
	\end{corollary}
	\begin{proof}
	   Let $\overline{\mathcal{K}}(\mathbb{U})_1=\hbox{ Prim}\big(\overline{\mathcal{K}}(\mathbb{U})\big)=\{ P\in \overline{\mathcal{K}}(\mathbb{U}), \delta(P)=0 \}$, and let $\overline{\mathcal{K}}(\mathbb{U})_n=\{ P\in \overline{\mathcal{K}}(\mathbb{U}), \delta(P)\in \sum \limits_{0<k<n}\overline{\mathcal{K}}(\mathbb{U})_k\otimes \overline{\mathcal{K}}(\mathbb{U})_{n-k}  \}$, we notice that $\overline{\mathcal{K}}(\mathbb{U})=\bigcup \limits_{k>0}\overline{\mathcal{K}}(\mathbb{U})_k$, i.e.
	   the vector space $(\overline{\mathcal{K}}(\mathbb{U}), \delta)$ is connected. By appling Corollary 3.9 in \cite{L. Muriel}, we obtain that $\overline{\mathcal{K}}(\mathbb{U})$ is generated as a pre-Lie algebra by $\hbox{ Prim}\big(\overline{\mathcal{K}}(\mathbb{U})\big)$.
	\end{proof}
	\begin{example}
	Here are the posets in $\hbox{ Prim}\big(\overline{\mathcal{K}}(\mathbb{U})\big)$ up to four vertices:
	$$
	\fcolorbox{white}{white}{
\scalebox{0.7}
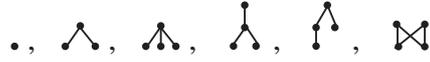
{
  \begin{picture}(5,17) (218,-263)
    \SetWidth{1.0}
    \SetColor{Black}
    \Vertex(226,-260){2}
  \end{picture}
}}\hspace{0.07cm}, 
\hspace{0.2cm}
	\fcolorbox{white}{white}{
\scalebox{0.7}{
  \begin{picture}(5,17) (218,-263)
    \SetWidth{1.0}
    \SetColor{Black}
    \Vertex(210,-260){2}
    \Vertex(226,-260){2}
    \Vertex(218,-249){2}
    \Line(210,-261)(219,-248)
    \Line(226,-260)(218,-250)
  \end{picture}
}}\hspace{0.07cm}, 
\hspace{0.2cm}\fcolorbox{white}{white}{
\scalebox{0.7}{
  \begin{picture}(5,17) (218,-263)
    \SetWidth{1.0}
    \SetColor{Black}
    \Vertex(210,-260){2}
     \Vertex(218,-260){2}
    \Vertex(226,-260){2}
    \Vertex(218,-249){2}
    \Line(210,-261)(219,-248)
    \Line(218,-260)(218,-248)
    \Line(226,-260)(218,-250)
  \end{picture}
}}\hspace{0.07cm},
\fcolorbox{white}{white}{
\scalebox{0.7}{
  \begin{picture}(18,27) (173,-248)
    \SetWidth{1.0}
    \SetColor{Black}
    \Vertex(177,-245){2}
    \Vertex(189,-245){2}
    \Vertex(183,-235){2}
    \Vertex(183,-223){2}
    \Line(177,-244)(182.3,-236.5)
    \Line(190,-244)(184,-235)
    \Line(183,-234)(183,-224)
  \end{picture}
}}, 
\fcolorbox{white}{white}{
\scalebox{0.7}{
  \begin{picture}(15,27) (174,-248)
    \SetWidth{1.0}
    \SetColor{Black}
    \Vertex(177,-245){2}
    \Vertex(183,-223){2}
    \Vertex(177,-235){2}
    \Vertex(187,-235){2}
    \Line(177,-245)(177,-234)
    \Line(177,-234)(184,-222)
    \Line(188,-235)(183,-223)
  \end{picture}
}},
\fcolorbox{white}{white}{
\scalebox{0.7}{
  \begin{picture}(26,39) (415,-324)
    \SetWidth{1.0}
    \SetColor{Black}
    \Vertex(433,-321){2}
    \Vertex(433,-309){2}
    \Vertex(418,-309){2}
    \Vertex(419,-321){2}
    \Line(418.5,-321)(418.5,-309)
    \Line(433,-322)(417,-308)
    \Line(433,-309)(417,-322)
    \Line(433,-322)(433,-309)
  \end{picture}
}}$$
	\end{example}
	\subsection{Duality relation between $\circleright$ and $\delta$}
	In this subsection, we prove that the NAP product $\circleright$ and the NAP coproduct $\delta$ are dual to each other modulo a symmetry factor.
	\begin{definition}
	Let $G$ be a group acting on $X$. For every $x\in X$:\\
	- we denote by $G\cdot x:=\{ g\cdot x, g\in G  \}$ the orbit of $x$,\\
	- we denote by $G_{x}=\{ g\in G, g\cdot x=x \}$ the stabilizer subgroup of $G$ with respect to $x$.
	\end{definition}
	\begin{proposition}
	Let $G$ be a group acting on $X$, if $G$ and $X$ is finite, then the orbit-stabilizer theorem, together with Lagrange’s theorem
\cite[Theorem 3.9]{acg..0}, gives:
    \begin{equation*}
        {\displaystyle |G\cdot x|=[G_{x}\,:\,G]=\dfrac{|G|}{|G_{x}|}}.
    \end{equation*}
In particular, the cardinal of the orbit is a divisor of the group order.
	\end{proposition}
\begin{definition}
For any poset $P$ on a finite set $X$,
    we denote by $\hbox{\hbox{Aut}}(P)$ the subgroup of permutations of $X$ which are homeomorphisms with respect to $P$. The symmetry factor is defined by $\sigma(P)=| \hbox{\hbox{Aut}}(P) |$.
    Let $X_1, X_2$ two finite sets, we define the linear map $\langle, \rangle:\mathbb{U}_{X_1}\otimes \mathbb{U}_{X_2}\longrightarrow \mathbb{K}$ by:
\[\langle Q,R\rangle=\begin{cases}
\sigma(Q)\mbox{ if }Q\approx R,\\
0\mbox{ otherwise}.
\end{cases}\]
In other terms, $\langle Q,R\rangle$ is the number of isomorphisms between $Q$ and $R$. 
    \end{definition}
     \begin{theorem}\label{thm1}
        Let $P, Q$ and $R$ three finite connected posets. We have the following identity
    $$\langle \delta(P), Q\otimes R\rangle=\frac{1}{|min(P)|}\langle P, Q\circleright R\rangle.$$
    \end{theorem}
   \begin{proof}
    Let $P, Q$ and $R$ three finite connected posets.
       \begin{align*}
\langle \delta(P),Q\otimes R\rangle&=\frac{1}{|\min(P)|} \sum_{\substack{I\circledcirc P\\I\approx Q\\ 
P\setminus I \approx R}}  \sigma(Q)\sigma(R)=\frac{1}{|\min(P)|} |A|
\end{align*}
with 
\[A=\{(I,\sigma_1,\sigma_2)\mid I\circledcirc P,\: \sigma_1:I\stackrel{\sim}{\longrightarrow} Q, \sigma_2:P\setminus I
\stackrel{\sim}{\longrightarrow} R\}\]
and:
\begin{align*}
\langle P,Q\circleright R\rangle&=\sum_{\substack{v\in \hbox{ \tiny{min}}(R)\\ Q\searrow_v R\approx P}}\sigma(P)=|B|
\end{align*}
with
\[B=\{(v,\sigma)\mid v\in \min(R),\: \sigma:P\stackrel{\sim}{\longrightarrow}Q\searrow_vR\}.\]

Let us define now a map $\phi:A\longrightarrow B$. Let $(I,\sigma_1,\sigma_2)\in A$. We put $w=I_-$.
As $I\circledcirc P$, $w\in \min(P\setminus I)$, so $v=\sigma_2(w)\in \min(R)$. As $I\circledcirc P$,
$P=I\searrow_w (P\setminus I)$, so we obtain an isomorphism $\sigma:P\longrightarrow Q\searrow_v R$
by taking
\[\sigma(x)=\begin{cases}
\sigma_1(x)\mbox{ if }x\in I,\\
\sigma_2(x)\mbox{ otherwise}.
\end{cases}\]
We then put $\phi(I,\sigma_1,\sigma_2)=(v,\sigma)$. 

Now, we define a map $\psi:B\longrightarrow A$. If $(v,\sigma)\in B$, we put $I=\sigma^{-1}(Q)$. 
As $Q\circledcirc Q\searrow_v R$, $I\circledcirc P$. Moreover, $\sigma_1=\sigma_{\mid I}$
is a graph isomorphism from $I$ to $Q$ and $\sigma_2=\sigma_{\mid P\setminus I}$ is a graph isomorphism
from $P\setminus I$ to $R$. We put $\psi(v,\sigma)=(I,\sigma_1,\sigma_2)$. 

Let $(v,\sigma)\in B$. We put $\psi(v, \sigma)=(I,\sigma_1,\sigma_2)$ and $\phi\circ \psi(v,\sigma)=(v',\sigma')$. 
Then, $I=\sigma^{-1}(Q)$ and $w'=I_-=w$, so $v'=\sigma(w)=v$. Moreover, 
\begin{align*}
\sigma'_{\mid I}&=\sigma_1=\sigma_{\mid I},\\
\sigma'_{\mid P\setminus I}&=\sigma_2=\sigma_{\mid P\setminus I},
\end{align*}
so $\sigma'=\sigma$. Hence, $\phi\circ \psi=\id_B$. 

Let $(I,\sigma_1,\sigma_2)\in A$. We put $\phi(I,\sigma_1,\sigma_2)=(v,\sigma)$ and  
$\psi\circ \phi(I,\sigma_1,\sigma_2)=(I',\sigma'_1,\sigma'_2)$. Then $I'=\sigma^{-1}(Q)=I$,
by construction of $\sigma$. Even more, 
\begin{align*}
\sigma_1'&=\sigma_{\mid I}=\sigma_1,\\
\sigma_2'&=\sigma_{\mid P\setminus I}=\sigma_2.
\end{align*}
So $\psi\circ \phi=\id_A$.
Finally, $A$ and $B$ are in bijection and we obtain:
\[\langle \delta(P),Q\otimes R\rangle=\frac{1}{|\min(P)|} \langle P,Q\circleright R\rangle.\] 
    \end{proof}
    \begin{remark}
As mentioned in the paper \cite{Moh. twisted}, the binary product $\nearrow$ on the species $\mathbb{U}$ of finite connected posets, defined by $P\nearrow Q:=j\big( j(P)\searrow j(Q) \big)$
for any pair $(P, Q)$ of finite connected posets,
  is a pre-Lie product, where $j$ is the involution which transforms $\le$ into $\ge$. Similarly,
 $$P \circleddotright Q:=j\big( j(P)\circleright j(Q) \big), \hspace{1cm}(\hbox{ resp. } \hspace{0.5cm} \tilde{\delta}= (j \otimes j)\delta \circ j)$$
  endows the species $\mathbb{U}$ with a twisted NAP algebra (resp. twisted connected NAP coalgebra) structure. 
    \end{remark}
    \textbf{Notation.} Let $E$ be any finite set, and let $\pi, \rho$ be two partitions of $E$. We denote by $j(\pi, \rho)$ the rational number
    $$j(\pi, \rho):=\frac{1}{|E|}\sum \limits_{\underset{\beta \hbox{ \tiny{is a block of }} \rho}{\alpha \hbox{ \tiny{is a block of }}} \pi} |\alpha \cap \beta|\frac{|\beta|}{|\alpha|}.$$
    
    \begin{proposition}\label{theorem.}
    Let $P, Q$ and $R$ three finite connected posets. We denote by $E=\{ v, v\in min(R)\mid P\approx Q\searrow_v R  \}$. Let $\pi$ be the partition of $E$ into Aut(P)-orbits, and let $\rho$ be the partition of $E$ into Aut(R)-orbits.
   We have the following identity
    $$j(\pi, \rho)=1.$$
    \end{proposition}
    
    \begin{proof}
   Let $P, Q$ and $R$ be three finite connected posets. In fact: firstly,
    \begin{align*}
        \langle P, Q\circleright R\rangle &=\sum_{v\in \tiny{\hbox{min}}(R)}\langle P, Q\searrow_v R\rangle \\
        &=\sum \limits_{\underset{P \approx Q\searrow_v R}{v\in \tiny{\hbox{min}}(R)}} \sigma(P)\\
        &=\sigma(P)|E|.
    \end{align*}
    On the other hand 
    \begin{align*}
        \langle \delta(P), Q\otimes R\rangle &=\frac{1}{|min(P)|}\sum_{I\circledcirc P} \langle I, Q\rangle \langle P\backslash I, R\rangle\\
        &=\frac{1}{|min(P)|}\sum \limits_{\underset{I\approx Q,\, P\backslash I\approx R}{I\circledcirc P}} \sigma(Q)\sigma(R)\\
        &=\frac{1}{|min(R)|}\sum \limits_{\underset{I\approx Q,\, P\backslash I\approx R}{I\circledcirc P}} \sigma(Q)\sigma(R)\\
        &=\frac{1}{|min(R)|}\sum_{v\in \tiny{\hbox{min}}(R)}\sum \limits_{\underset{I\approx Q,\, P\backslash I\approx R}{I\circledcirc P,\, I_-=\{v\}}} \sigma(Q)\sigma(R)\\
        &= \frac{1}{|min(R)|}\sum_{v\in \tiny{\hbox{min}}(R)}N_v(P, Q, R) \sigma(Q)\sigma(R),
    \end{align*}
    where $N_v(P, Q, R):=$ the number of branches $I$ of $P$ above $v$ isomorphic to $Q$ such that $P\backslash I$ isomorphic to $R$.\\
      We notice that, for all $v\in min(R)$: 
    \begin{itemize}
        \item if $P\napprox Q\searrow_v R$, then $N_v(P, Q, R)=0$,
        \item if $P\approx Q\searrow_v R$, then $N_v(P, Q, R)=N_v(Q, R):=$ the number of branches of $Q\searrow_v R$ above $v$ isomorphic to $Q$.
    \end{itemize}
    Accordingly 
     \begin{align*}
        \langle \delta(P), Q\otimes R\rangle &=\frac{1}{|min(R)|}\Big( \sum \limits_{\underset{P\approx Q\searrow_v R}{v\in \tiny{\hbox{min}}(R)}} N_v(P, Q, R) \sigma(Q)\sigma(R)  + \sum \limits_{\underset{P\ne Q\searrow_v R}{v\in \tiny{\hbox{min}}(R)}} N_v(P, Q, R) \sigma(Q)\sigma(R) \Big)\\
        &= \frac{1}{|min(R)|} \sum \limits_{\underset{P\approx Q\searrow_v R}{v\in \tiny{\hbox{min}}(R)}} N_v(P, Q, R) \sigma(Q)\sigma(R)\\
        &=\frac{1}{|min(R)|} \sum \limits_{\underset{P\approx Q\searrow_v R}{v\in \tiny{\hbox{min}}(R)}} N_v(Q, R) \sigma(Q)\sigma(R).
    \end{align*}
    By using the orbit-stabilizer theorem, as well as Lagrange's theorem for the group $\big( Aut(R), \circ \big)$ endowed with the law of composition, we therefore have:
  $$|Aut(R)\cdot v|=\frac{|Aut(R)|}{|Aut(R)_v|},\hspace{0.5cm} i.e.\hspace{0.5cm} |Aut(R)\cdot v|=\frac{\sigma(R)}{\sigma_v(R)}.$$
  We notice that, for all $v\in \hbox{min}(R)$: $\sigma_v(Q\searrow_v R)=\sigma(Q)\sigma_v(R)N_v(Q, R)$.\\
  Then 
  \begin{align*}
        \langle \delta(P), Q\otimes R\rangle &= \frac{1}{|min(R)|} \sum \limits_{\underset{}{v\in E}}  \sigma(Q)\sigma_v(R)|Aut(R)\cdot v|N_v(Q, R)\\
        &=\frac{1}{|min(R)|} \sum \limits_{\underset{}{v\in E}}  \sigma_v(P)|Aut(R)\cdot v|\\
        &=\frac{1}{|min(R)|}\sum \limits_{\underset{}{v\in E}} \sigma(P)\frac{|Aut(R)\cdot v|}{|Aut(P)\cdot v|}\\
        &=\frac{1}{|min(R)|}\sigma(P) S,
    \end{align*}
    where $$S= \sum \limits_{\underset{}{v\in E}} \frac{|Aut(R)\cdot v|}{|Aut(P)\cdot v|}.$$
    Let $\pi$ be the partition of $E$ into $Aut(P)$-orbits, and let $\rho$ be the partition of $E$ into $Aut(R)$-orbits. We denote by $\alpha \in \pi$ (respectively, $\beta \in \rho$), if $\alpha$ is a block of $\pi$ (respectively, if $\beta$ is a block of $\rho$).\\
    So
    \begin{align*}
        S&= \sum \limits_{\underset{}{v\in E}} \frac{|Aut(R)\cdot v|}{|Aut(P)\cdot v|}\\
        &=\sum \limits_{\underset{}{v\in E}}\sum \limits_{\underset{\beta \hbox{ \tiny{is a Aut(R)-orbit of v} }}{\alpha \hbox{ \tiny{is a Aut(P)-orbit of v} }}} \frac{|\beta|}{|\alpha|}\\
        &=\sum \limits_{\underset{\beta \hbox{ \tiny{is a Aut(R)-orbit} }}{\alpha \hbox{ \tiny{is a Aut(P)-orbit} }}} |\alpha \cap \beta|\frac{|\beta|}{|\alpha|}\\
        &=j(\pi, \rho),
    \end{align*}
    then 
    $\langle \delta(P), Q\otimes R\rangle =\frac{1}{|min(R)|}\sigma(P)j(\pi, \rho)$. Correspondingly,
    $$\langle \delta(P), Q\otimes R\rangle=\langle P, Q\circleright R\rangle \frac{j_{P, R}(\pi, \rho)}{|min(P)|}.$$
    
     By appling Theorem \ref{thm1}, we obtain: 
        $$\langle \delta(P), Q\otimes R\rangle=\frac{1}{|min(P)|}\langle P, Q\circleright R\rangle=\langle P, Q\circleright R\rangle \frac{j(\pi, \rho)}{|min(P)|}.$$
    
    Conclusion: for $P, Q$ and $R$ three finite connected posets, if we denote by $E=\{ v, v\in min(R)\mid P\approx Q\searrow_v R  \}$, and for $\pi$ (resp. $\rho$) the partition of $E$ into $Aut(P)$-orbits (resp. the partition of $E$ into $Aut(R)$-orbits). Then we obtain
    $$j(\pi, \rho)=1.$$
    \end{proof}
    \begin{remark}
   We do not have $j(\pi, \rho)=1$ in general for two partitions of a finite set $E$.
Indeed:
let $E$ be the finite set $\{1,2,3,4\}$, $\pi=\{ \{1,2\}, \{3,4\}  \}$ and $\rho=\{ \{1,2,3\},\{4\} \}$. Then we obtain $j(\pi, \rho)=\frac{5}{4}$ and $j(\rho, \pi)=1$.
This interesting index does not seem to be present in the literature.
    \end{remark}
    

\subsection{Free Lie algebras of finite posets}
F. Chapoton in \cite{F. Chapoton.} showed that free pre-Lie algebras are free as Lie algebras i.e., if $(V, \rhd)$
is a free pre-Lie algebra. Then it is a free Lie algebra for the
Lie bracket $[, ]$ defined by $[a, b]=a\rhd b -b\rhd a$. This result has been obtained before with different methods by L. Foissy \cite[Theorem 8.4]{L. Foissy}.
\begin{proposition}\cite{L. Foissy, F. Chapoton.}
For any vector space $V$, the free pre-Lie algebra on $V$ is isomorphic as a
Lie algebra to the free Lie algebra on some set $X(V)$ of generators.
\end{proposition}
\textbf{Application}\cite{acg1} Let $H$ be a Hopf algebra and 
 $H^0$ the graded dual of $H$. The primitive element algebra of the graded dual $H^0$ with the bracket $[,]$ is a Lie algebra. We denote by $\star$ the Grossman-Larson product 
on the dual of $H$ \cite{Moh. twisted}. This product restricted to $H^0$ is the graded dual of the coproduct $\Delta_{\searrow}$\cite{Moh. twisted}.
Applying the general setting above to the case $H=\overline{\mathcal{K}}(\mathbb{P})$ we are studying, we obtain:\\
 $\big(\overline{\mathcal{K}}(\mathbb{U}), \searrow \big)=\big(Prim(H^0), \searrow \big)$ is a free pre-Lie algebra, then it is also a free Lie algebra,
therefore $U(Prim H^0)=\big(S(PrimH^0), \star \big)$ is a free associative algebra.
Then, $\big(S(PrimH^0), \star\big)^0=(H, \Delta_{\searrow}) $ is a coassociative cofree coalgebra.
\section{Generalization to finite topological spaces}\label{Generalized results}


All the results in this paper remain true for the  finite connected topological spaces, with a small change on the definition of the coproduct $\delta$.
Indeed: let $\mathcal{T}=(X, \le_{\mathcal{T}}) $ be a finite connected topological space. We define
\begin{eqnarray*}
		\delta(\mathcal{T})=\frac{1}{|min(\mathcal{T})|}\sum_{Y \circledcirc  \mathcal{T}} \mathcal{T}_{|Y} \otimes \mathcal{T}_{|X\backslash Y},
	\end{eqnarray*}
	where $Y\circledcirc \mathcal{T}$ means that $Y\in \mathcal{T}$ such that:
	\begin{itemize}
	\item $\overline{Y}_-$ is a singleton included in $\hbox{min}(\overline{\mathcal{T}})$,
    \item and $Y$ is a branch of the set $\{x\in X, \hbox{ such that for all } y\in Y_-, \hbox{ we have } y<_{\mathcal{T}} x \hbox{ and } y \ngeq x \}$.
\end{itemize}
Here, the set $Y_-$ equal to the space $ \{x\notin Y, \hbox{ there exists } y \in Y \hbox{ such that } x \le_{\mathcal{t}} y \} $, $\overline{\mathcal{T}}=(X/\sim_{\mathcal{T}}, \le_{})$ is the poset defined on $X/\sim_{\mathcal{T}}$ corresponding to the partial order $\le$ induced by the quasi-order $\le_{\mathcal{T}}$, and $\overline{Y_-}=Y_-/\sim_{\mathcal{T}_{|Y_-}}$.
Let $\hbox{\hbox{Aut}}(\mathcal{T})$ be the subgroup of permutations of $X$ which are homeomorphisms with respect to $\mathcal{T}$, 
and let $G=\{ g\in Aut(\mathcal{T}), g\cdot \overline{v}=\overline{v} \hbox{ for all } v\in X \}$.
We notice that $G\subseteq Aut(\mathcal{T})$, $G$ is a normal subgroup and $Aut(\overline{\mathcal{T}})=Aut(\mathcal{T})/G$.\\

	We notice that the results obtained in Lemma \ref{lem}  also work for $Y, Z\circledcirc \mathcal{T}$. From $Y, Z\circledcirc \mathcal{T}$ and Lemma \ref{lem}, we get that both $\overline{Y}, \overline{Z} \circledcirc \overline{\mathcal{T}}$, then ($\overline{Y}=\overline{Z} \hbox{ or }\overline{Y}\cap \overline{Z}=\emptyset$), and $\hbox{min}(\overline{\mathcal{T}}_{|X\backslash Y})=\hbox{min}(\overline{\mathcal{T}})$. Then $Y\cap Z= \emptyset$ and $\hbox{min}(\mathcal{T}_{|X\backslash Y})=\hbox{min}(\mathcal{T})$.\\
Hence to show that the proposition \ref{prop} also works in topological cases, it suffices to change $ I, J \circledcirc P $ by $ Y, Z \circledcirc \mathcal{T} $ in the proof of proposition \ref{theorem.}.
To show that the result of Theorem \ref{Theorem} also works in connected finite topological spaces, we change $ I \circledcirc P $ by $ Y \circledcirc \mathcal{T} $ in the Proof \ref{prv}.\\

 Finally we prove the existence of a relation between $ \delta $ and $ \circleright $ in the topological case. Exactly: for
   $\mathcal{T}, \mathcal{T}^{'}$ and $\mathcal{T}^{''}$ three finite connected topological spaces, we have
  $$\langle \delta(\mathcal{T}),\mathcal{T}'\otimes \mathcal{T}''\rangle= \frac{1}{n|\min(\mathcal{T})|}  \langle \mathcal{T}, \mathcal{T}'\circleright \mathcal{T}''\rangle,$$
  where $n$ is the number of items in any bag $\overline{v}\in \hbox{ min}(\overline{\mathcal{T}''})$, such as $\mathcal{T}\approx \mathcal{T}'\searrow_v \mathcal{T}''$.\\
 The proof is similar to the proof of Theorem \ref{thm1}. Indeed:
   \begin{align*}
\langle \delta(\mathcal{T}),\mathcal{T}'\otimes \mathcal{T}''\rangle&=\frac{1}{|\min(\mathcal{T})|} \sum_{\substack{Y\circledcirc \mathcal{T}\\\mathcal{T}_{|Y}\approx \mathcal{T}',\\ 
\mathcal{T}_{|X \setminus Y} \approx \mathcal{T}''}}  \sigma(\mathcal{T}')\sigma(\mathcal{T}'')=\frac{1}{|\min(\mathcal{T})|} |A|
\end{align*}
with 
\[A=\{(Y,\sigma_1,\sigma_2)\mid Y\circledcirc \mathcal{T},\: \sigma_1:\mathcal{T}_{|Y}\stackrel{\sim}{\longrightarrow} \mathcal{T}', \sigma_2:\mathcal{T}_{|X \setminus Y}
\stackrel{\sim}{\longrightarrow} \mathcal{T}''\}.\]
     On the other hand 
 \begin{align*}
        \langle \mathcal{T}, \mathcal{T}'\circleright \mathcal{T}''\rangle &=\sum_{v\in \tiny{\hbox{min}}(\mathcal{T}'')}\langle \mathcal{T}, \mathcal{T}'\searrow_v \mathcal{T}''\rangle \\
        &=\sum \limits_{\underset{\mathcal{T}\approx \mathcal{T}'\searrow_v \mathcal{T}''}{v\in \tiny{\hbox{min}}(\mathcal{T}'')}} \sigma(\mathcal{T})=|B|,
    \end{align*}
    with
\[B=\{(v,\sigma)\mid v\in \min(\mathcal{T}''),\: \sigma:\mathcal{T}\stackrel{\sim}{\longrightarrow}\mathcal{T}' \searrow_v \mathcal{T}''\}.\]
  Let \[C=\{(u,\sigma)\mid u=\overline{v}, v\in \min(\mathcal{T}''),\: \sigma:\mathcal{T}\stackrel{\sim}{\longrightarrow}\mathcal{T}' \searrow_v \mathcal{T}''\}.\]
  We notice that for all $u, u'\in C$, we have $|u|=|u'|$. Hence $\frac{|B|}{|C|}$ is the number of items in any bag $\overline{v}\in \hbox{ min}(\overline{\mathcal{T}''})$, such as $\mathcal{T}\approx \mathcal{T}'\searrow_v \mathcal{T}''$.
    We define
    	\begin{eqnarray*}
		\ f:A&  \longrightarrow  C\\
		(Y,\sigma_1,\sigma_2)&\longmapsto& (u,\sigma),
	\end{eqnarray*}
 where $u=\overline{v}=\overline{\sigma_2(Y_{-})}$ and $\sigma$ is a graph isomorphism from $\mathcal{T}$ to $\mathcal{T}'\searrow_v \mathcal{T}''$ defined by
\[\sigma(x)=\begin{cases}
\sigma_1(x)\mbox{ if }x\in Y,\\
\sigma_2(x)\mbox{ otherwise}.
\end{cases}\]
 Now, we define a map
 \begin{eqnarray*}
		\ g:C&  \longrightarrow  A\\
		(u,\sigma)&\longmapsto& (Y,\sigma_1, \sigma_2),
	\end{eqnarray*}
	where $Y=X_{\sigma^{-1}(\mathcal{T}')}$, $\sigma_1=\sigma_{\mid \mathcal{T}_{\mid Y}}$
is a graph isomorphism from $\mathcal{T}_{\mid Y}$ to $\mathcal{T}'$ and $\sigma_2=\sigma_{\mid \mathcal{T}_{X\setminus Y}}$ is a graph isomorphism
from $\mathcal{T}_{\mid X\setminus Y}$ to $\mathcal{T}''$.\\
Let $(u, \sigma)\in C$, we put $g(u, \sigma)=(Y, \sigma_1, \sigma_2)$, where $Y=X_{\sigma^{-1}(\mathcal{T}')}$, $\sigma_1=\sigma_{\mid \mathcal{T}'}$ and $\sigma_2=\sigma_{\mid \mathcal{T}''}$.
Then $f\circ g(u, \sigma)=(u', \sigma')$, with $u'=\overline{w}=\overline{\sigma_2(Y_-)}$ and  $\sigma'$ is a graph isomorphism from $\mathcal{T}$ to $\mathcal{T}'\searrow_{w} \mathcal{T}''$ defined by
\[\sigma'(x)=\begin{cases}
\sigma_1(x)\mbox{ if }x\in Y,\\
\sigma_2(x)\mbox{ otherwise}.
\end{cases}\]
Then, $u'=\overline{\sigma_2(Y_-)}=u$ and $\sigma'=\sigma$. Hence $f\circ g=Id_C$.\\
Let $(Y, \sigma_1, \sigma_2)\in A$, we put $f(Y, \sigma_1, \sigma_2)=(u, \sigma)$, with $u=\overline{v}=\overline{\sigma_2(Y_-)}$ and $\sigma$ is a graph isomorphism from $\mathcal{T}$ to $\mathcal{T}'\searrow_{v} \mathcal{T}''$ defined by
\[\sigma(x)=\begin{cases}
\sigma_1(x)\mbox{ if }x\in Y,\\
\sigma_2(x)\mbox{ otherwise}.
\end{cases}\]
Then, $g\circ f(Y, \sigma_1, \sigma_2)=(Y', \sigma^{'}_1, \sigma^{'}_2)$, with $Y'=X_{\sigma^{-1}(\mathcal{T}')}$,  $\sigma^{'}_1=\sigma_{\mid \mathcal{T}_{\mid Y'}}$
is a graph isomorphism from $\mathcal{T}_{\mid Y'}$ to $\mathcal{T}'$ and $\sigma^{'}_2=\sigma_{\mid \mathcal{T}_{|X\setminus Y'}}$ is a graph isomorphism
from $\mathcal{T}_{\mid X\setminus Y'}$ to $\mathcal{T}''$.\\
Then, $Y'=X_{\sigma^{-1}(\mathcal{T}')}=Y$ and 
\[\begin{cases}
\sigma^{'}_1=\sigma_{\mid \mathcal{T}_{\mid Y'}}=\sigma_1,\\
\sigma^{'}_2=\sigma_{\mid \mathcal{T}_{|X\setminus Y'}}=\sigma_2.
\end{cases}\]
So $g\circ f= Id_A$.\\
Finally, $A$ and $C$ are in bijection.
 Hence
 $$\langle \delta(\mathcal{T}),\mathcal{T}'\otimes \mathcal{T}''\rangle=\frac{1}{|\min(\mathcal{T})|} |A|=\frac{1}{|\min(\mathcal{T})|} |C|= \frac{1}{|\min(\mathcal{T})|}|C|\frac{|B|}{|B|}= \frac{|C|}{|B||\min(\mathcal{T})|}  \langle \mathcal{T}, \mathcal{T}'\circleright \mathcal{T}''\rangle.$$

\textbf{Remark.}
All the results in this paper remain true for the posets (resp. topological spaces)  decorated by $F$, that is to say pairs $(\mathcal{A}, f)$, where $\mathcal{A}=(A, \le_{\mathcal{A}})$ is a poset (resp. topology) and $f:A\to F$ is a map.
\\
\\
\\

\textbf{Acknowledgements}: I would like to thanks Loïc Foissy and my Ph.D. thesis supervisors: Professor Dominique Manchon and Professor Ali Baklouti for their interest in this work and useful suggestions.\\

\textbf{Funding}: This work was supported by the University of Sfax, Tunisia.



\end{document}